\documentclass[12pt,letterpaper,american,english]{amsart}
\usepackage{charter}
\usepackage{helvet}
\usepackage{courier}
\usepackage[T1]{fontenc}
\usepackage[utf8x]{inputenc}
\usepackage{mathrsfs}
\usepackage{amsthm}
\usepackage{amsbsy}
\usepackage{amstext}
\usepackage{amssymb}
\usepackage[all]{xy}

\makeatletter

\pdfpageheight\paperheight
\pdfpagewidth\paperwidth

\numberwithin{equation}{section}
\numberwithin{figure}{section}
\theoremstyle{plain}
\newtheorem{thm}{\protect\theoremname}[section]
  \theoremstyle{definition}
  \newtheorem{defn}[thm]{\protect\definitionname}
  \theoremstyle{remark}
  \newtheorem*{rem*}{\protect\remarkname}
  \theoremstyle{remark}
  \newtheorem{rem}[thm]{\protect\remarkname}
  \theoremstyle{definition}
  \newtheorem{example}[thm]{\protect\examplename}
  \theoremstyle{plain}
  \newtheorem{cor}[thm]{\protect\corollaryname}
  \theoremstyle{plain}
  \newtheorem{prop}[thm]{\protect\propositionname}
  \theoremstyle{plain}
  \newtheorem{lem}[thm]{\protect\lemmaname}
  \theoremstyle{definition}
  \newtheorem{condition}[thm]{\protect\conditionname}

\usepackage{verbatim} 
\usepackage{eucal}
\usepackage{amsthm}
\usepackage[all]{xy}
\usepackage{manfnt}
\usepackage{float}
\usepackage{xr}
\SelectTips{cm}{}
\xyoption{line}
\usepackage{hyperref}
\makeatletter \newcommand{\xyR}[1]{%
\makeatletter \xydef@\xymatrixrowsep@{#1} \makeatother }
\makeatletter \newcommand{\xyC}[1]{%
\makeatletter \xydef@\xymatrixcolsep@{#1} \makeatother }
\mathcode`\:="603A 
\DeclareSymbolFont{rsfs}{U}{rsfs}{m}{n}
\DeclareSymbolFontAlphabet{\mathrf}{rsfs}

\makeatother

\usepackage{babel}
  \addto\captionsamerican{\renewcommand{\conditionname}{Condition}}
  \addto\captionsamerican{\renewcommand{\corollaryname}{Corollary}}
  \addto\captionsamerican{\renewcommand{\definitionname}{Definition}}
  \addto\captionsamerican{\renewcommand{\examplename}{Example}}
  \addto\captionsamerican{\renewcommand{\lemmaname}{Lemma}}
  \addto\captionsamerican{\renewcommand{\propositionname}{Proposition}}
  \addto\captionsamerican{\renewcommand{\remarkname}{Remark}}
  \addto\captionsamerican{\renewcommand{\theoremname}{Theorem}}
  \addto\captionsenglish{\renewcommand{\conditionname}{Condition}}
  \addto\captionsenglish{\renewcommand{\corollaryname}{Corollary}}
  \addto\captionsenglish{\renewcommand{\definitionname}{Definition}}
  \addto\captionsenglish{\renewcommand{\examplename}{Example}}
  \addto\captionsenglish{\renewcommand{\lemmaname}{Lemma}}
  \addto\captionsenglish{\renewcommand{\propositionname}{Proposition}}
  \addto\captionsenglish{\renewcommand{\remarkname}{Remark}}
  \addto\captionsenglish{\renewcommand{\theoremname}{Theorem}}
  \providecommand{\conditionname}{Condition}
  \providecommand{\corollaryname}{Corollary}
  \providecommand{\definitionname}{Definition}
  \providecommand{\examplename}{Example}
  \providecommand{\lemmaname}{Lemma}
  \providecommand{\propositionname}{Proposition}
  \providecommand{\remarkname}{Remark}
\providecommand{\theoremname}{Theorem}

\begin{document}

\title{$\mathfrak{S}$-coalgebras determine fundamental groups}

\author{Justin R. Smith}



\curraddr{Department of Mathematics\\
Drexel University\\
Philadelphia,~PA 19104}

\email{jsmith@drexel.edu}

\urladdr{http://vorpal.math.drexel.edu}

\maketitle
\global\long\def\ring{\mathbb{Z}}
\global\long\def\integers{\mathbb{Z}}
\global\long\def\betabar{\bar{\beta}}
 \global\long\def\desusp{\downarrow}
\global\long\def\susp{\uparrow}
\global\long\def\cobar{\mathcal{F}}
\global\long\def\coend{\mathrm{CoEnd}}
\global\long\def\ainfty{A_{\infty}}
\global\long\def\coassoc{\mathrm{Coassoc}}
\global\long\def\trm{\mathrm{T}}
\global\long\def\tfr{\mathfrak{T}}
\global\long\def\tabbr{\hat{\trm}}
\global\long\def\Tabbr{\hat{\tfr}}
\global\long\def\afr{\mathfrak{A}}
\global\long\def\homz{\mathrm{Hom}_{\ring}}
\global\long\def\zend{\mathrm{End}}
\global\long\def\rs#1{\mathrm{R}S_{#1 }}
\global\long\def\forgetful#1{\lceil#1\rceil}
\global\long\def\highprod#1{\bar{\mu}_{#1 }}
\global\long\def\slength#1{|#1 |}
\global\long\def\barcs{\bar{\mathcal{B}}}
\global\long\def\ubarcs{\mathcal{B}}
\global\long\def\zs#1{\ring S_{#1 }}
\global\long\def\homzs#1{\mathrm{Hom}_{\ring S_{#1 }}}
\global\long\def\zpi{\mathbb{Z}\pi}
\global\long\def\D{\mathfrak{D}}
\global\long\def\ahat{\hat{\mathfrak{A}}}
\global\long\def\cbar{{\bar{C}}}
\global\long\def\cf#1{\mathcal{C}(#1 )}
\global\long\def\ddelta{\dot{\Delta}}
\global\long\def\dimlimiter{\triangleright}
\global\long\def\coalgcat{\mathrf S_{0}}
\global\long\def\hcoalgcat{\mathrf{S}}
\global\long\def\ircoalgcat{\mathrf I_{0}}
\global\long\def\bircoalgcat{\mathrf{I}_{0}^{+}}
\global\long\def\hircoalgcat{\mathrf I}
\global\long\def\dcoalgcat{\mathrm{ind}-\coalgcat}
\global\long\def\chaincat{\mathbf{Ch}}
\global\long\def\coll{\mathrm{Coll}}
\global\long\def\bchaincat{\mathbf{Ch}_{0}}
\global\long\def\ilimit{\varprojlim\,}
\global\long\def\bigboxtimes{\mathop{\boxtimes}}
\global\long\def\dlimit{\varinjlim\,}
\global\long\def\coker{\mathrm{{coker}}}
\global\long\def\icoalgcat{\mathrm{pro}-\mathrf{S}_{0}}
\global\long\def\iircoalgcat{\mathrm{pro-}\ircoalgcat}
\global\long\def\dircoalgcat{\mathrm{ind-}\ircoalgcat}
\global\long\def\core#1{\left\langle #1\right\rangle }
\global\long\def\ilimitder{\varprojlim^{1}\,}
\global\long\def\pcoalg#1#2{P_{\mathcal{#1}}(#2) }
\global\long\def\pcoalgf#1#2{P_{\mathcal{#1}}(\forgetful{#2}) }
\global\long\def\coequalizer{\mathop{\mathrm{coequalizer}}}

\global\long\def\mainoperad{\mathcal{H}}
\global\long\def\cone#1{\mathrm{Cone}(#1)}

\global\long\def\im{\operatorname{im}}

\global\long\def\lcell{L_{\mathrm{cell}}}
\global\long\def\ccoalgcat{\mathrf S_{\mathrm{cell}}}

\global\long\def\fc#1{\mathrm{hom}(\bigstar,#1)}
\global\long\def\coS{\mathbf{coS}}
\global\long\def\cocell{\mathbf{co}\ccoalgcat}

\global\long\def\ccoalgcat{\mathrf S_{\mathrm{cell}}}

\global\long\def\spaces{\mathbf{SS}}

\global\long\def\pgam{\tilde{\Gamma}}
\global\long\def\pz{\tilde{\integers}}

\global\long\def\moore#1{\{#1\}}

\global\long\def\ints{\mathbb{Z}}

\global\long\def\finite{\mathcal{F}}

\global\long\def\finiteop{\finite^{\mathrm{op}}}

\global\long\def\syms{\mathbf{SS}}

\global\long\def\ordered{\mathbf{\Delta}}

\global\long\def\sets{\mathbf{Set}}

\global\long\def\colim{\operatorname{colim}}

\newdir{ >}{{}*!/-5pt/@{>}}

\global\long\def\treal#1{\mathcal{T}(\bigstar,#1)}

\global\long\def\rats{\mathbb{Q}}

\global\long\def\img{\operatorname{im}}

\global\long\def\tmap#1{\mathrm{T}_{#1}}

\global\long\def\Tmap#1{\mathfrak{T}_{#1}}

\global\long\def\glist#1#2#3{#1_{#2},\dots,#1_{#3}}

\global\long\def\blist#1#2{\glist{#1}1{#2}}

\global\long\def\enlist#1#2{\{\blist{#1}{#2}\}}

\global\long\def\tlist#1#2{\tmap{\blist{#1}{#2}}}

\global\long\def\Tlist#1#2{\Tmap{\blist{#1}{#2}}}

\global\long\def\nth#1{\mbox{#1}^{\mathrm{th}}}

\global\long\def\tunder#1#2{\tmap{\underbrace{{\scriptstyle #1}}_{#2}}}

\global\long\def\Tunder#1#2{\Tmap{\underbrace{{\scriptstyle #1}}_{#2}}}

\global\long\def\tunderi#1#2{\tunder{1,\dots,#1,\dots,1}{#2^{\mathrm{th}}\ \mathrm{position}}}

\global\long\def\Tunderi#1#2{\Tunder{1,\dots,#1,\dots,1}{#2^{\mathrm{th}}\,\mathrm{position}}}

\global\long\def\chaincat{\mathbf{Ch}}

\global\long\def\chaincatp{\chaincat_{0}}

\global\long\def\simpc{\mathbf{SC}}

\global\long\def\s{\mathfrak{S}}

\global\long\def\pco{P_{\s}}

\global\long\def\lco{L_{\s}}

\global\long\def\ns#1{\mathcal{N}^{#1}}

\global\long\def\cfn#1{\mathcal{N}(#1)}

\global\long\def\nfc#1{\operatorname{hom}_{\mathbf{n}}(\bigstar,#1)}

\date{\today}
\begin{abstract}
In this paper, we extend earlier work by showing that if $X$ and
$Y$ are simplicial complexes (i.e. simplicial sets whose nondegenerate
simplices are determined by their vertices), an \emph{isomorphism}
$\cfn X\cong\cfn Y$ of $\s$-coalgebras implies that the 3-skeleton
of $X$ is weakly equivalent to the 3-skeleton of $Y$, also implying
that $\pi_{1}(X)=\pi_{1}(Y)$.
\end{abstract}

\section{Introduction}

In \cite{Smith:1994}, the author defined the functor $\cf *$ on
simplicial sets --- essentially the chain complex equipped with the
structure of a coalgebra over an operad $\s$. This coalgebra structure
determined all Steenrod and other cohomology operations. Since these
coalgebras are not \emph{nilpotent}%
\footnote{In a nilpotent coalgebra, iterated coproducts of elements ``peter
out'' after a finite number of steps. See\foreignlanguage{american}{
\cite[chapter~3]{operad-book} for the precise definition.}%
}\emph{ }they have a kind of ``transcendental'' structure that contains
much more information.

In section~\ref{sec:The-functor-cfn}, we define a variant of the
$\cf *$-functor, named $\cfn *$. It is defined for simplicial complexes
--- semi-simplicial sets whose simplices are uniquely determined by
their vertices. The script-N emphasizes that its underlying chain-complex
is \emph{normalized} and $\cf *$ can be views as an extension of
$\cfn X$ to general simplicial sets (see \cite{smith-cellular}).

In \cite{smith-cellular}, we showed that if $X$ and $Y$ are pointed,
reduced simplicial sets, then a quasi-isomorphism $\cf X\to\cf Y$
induces one of their $\ints$-completions $\ints_{\infty}X\to\ints_{\infty}Y$.
It follows that that the $\cf *$-functor determine a \emph{nilpotent
}space's weak homotopy type.

In the present paper, we extend this by showing:

\medskip{}

Corollary. \ref{cor:cellular-determines-pi1}. \emph{If $X$ and $Y$
are simplicial complexes with the property that there exists an isomorphism
\[
g:\cfn X\to\cfn Y
\]
then their 3-skeleta are isomorphic and 
\[
\pi_{1}(X)\cong\pi_{1}(Y)
\]
}

\medskip{}

This implies that the functors $\cf *$ and $\cfn *$ encapsulate
``non-abelian'' information about a simplicial set --- such as its
(possibly non-nilpotent) fundamental group. The requirement that $g$
be an \emph{isomorphism} is stronger than needed for this (but \emph{quasi}-isomorphism
is not enough). 

The proof actually requires $X$ and $Y$ to be simplicial \emph{complexes}
rather than general simplicial sets. The author conjectures that the
$\cf *$-functor determines the integral homotopy type of an arbitrary
simplicial set.

Since the transcendental portion of $\cf X$ can be mapped to a power
series ring (see the proof of lemma~\ref{lem:diagonals-linearly-independent}),
the analysis of this data may require methods of analysis and algebraic
geometry.

\section{Definitions and assumptions}
\begin{defn}
\label{defr:chaincat}Let $\chaincat$ denote the category of $\ints$-graded
$\ints$-free chain complexes and let $\chaincatp\subset\chaincat$
denote the subcategory of chain complexes concentrated in positive
dimensions. 

If $c\in\chaincat$, 
\[
C^{\otimes n}=\underbrace{C\otimes_{\ints}\otimes\cdots\otimes_{\ints}C}_{n\,\text{times }}
\]

\end{defn}
We also have categories of spaces:
\begin{defn}
\label{def:simplicial-set-complex}Let $\spaces$ denote the category
of simplicial sets and $\simpc$ that of simplicial \emph{complexes.}
A simplicial \emph{complex} is a simplicial set without degeneracies
(i.e., a semi-simplicial set) with the property that simplices are
uniquely determined by their vertices.\end{defn}
\begin{rem*}
Following \cite{may-finite}, we can define a simplicial set to have
\emph{Property~A} if every face of a nondegenerate simplex is nondegenerate.
Theorem~12.4.4 of \cite{may-finite} proves that simplicial sets
with property~A have \emph{second subdivisions} that are simplicial
complexes. The bar-resolution $\rs 2$ is an example of a simplicial
set that does \emph{not} have property~A.

On the other hand, it is well-known that \emph{all} topological spaces
are weakly homotopy equivalent to simplicial complexes --- see, for
example, Theorem~2C.5 and Proposition~4.13 of \cite{hatcher-alg-top}.
\end{rem*}
We make extensive use of the Koszul Convention (see~\cite{Gugenheim:1960})
regarding signs in homological calculations:
\begin{defn}
\label{def:koszul-1} If $f:C_{1}\to D_{1}$, $g:C_{2}\to D_{2}$
are maps, and $a\otimes b\in C_{1}\otimes C_{2}$ (where $a$ is a
homogeneous element), then $(f\otimes g)(a\otimes b)$ is defined
to be $(-1)^{\deg(g)\cdot\deg(a)}f(a)\otimes g(b)$. \end{defn}
\begin{rem}
If $f_{i}$, $g_{i}$ are maps, it isn't hard to verify that the Koszul
convention implies that $(f_{1}\otimes g_{1})\circ(f_{2}\otimes g_{2})=(-1)^{\deg(f_{2})\cdot\deg(g_{1})}(f_{1}\circ f_{2}\otimes g_{1}\circ g_{2})$.\end{rem}
\begin{defn}
\label{def:homcomplex-1}Given chain-complexes $A,B\in\chaincat$
define
\[
\homz(A,B)
\]
to be the chain-complex of graded $\ring$-morphisms where the degree
of an element $x\in\homz(A,B)$ is its degree as a map and with differential
\[
\partial f=f\circ\partial_{A}-(-1)^{\deg f}\partial_{B}\circ f
\]
As a $\ring$-module $\homz(A,B)_{k}=\prod_{j}\homz(A_{j},B_{j+k})$.\end{defn}
\begin{rem*}
Given $A,B\in\mathbf{Ch}^{S_{n}}$, we can define $\homzs n(A,B)$
in a corresponding way.\end{rem*}
\begin{defn}
\label{def:tmap} Let $\alpha_{i}$, $i=1,\dots,n$ be a sequence
of nonnegative integers whose sum is $|\alpha|$. Define a set-mapping
 of symmetric groups 
\[
\tlist{\alpha}n:S_{n}\to S_{|\alpha|}
\]
 as follows:
\begin{enumerate}
\item for $i$ between 1 and $n$, let $L_{i}$ denote the length-$\alpha_{i}$
integer sequence: 
\item ,where $A_{i}=\sum_{j=1}^{i-1}\alpha_{j}$ --- so, for instance, the
concatenation of all of the $L_{i}$ is the sequence of integers from
1 to $|\alpha|$; 
\item $\tlist{\alpha}n(\sigma)$ is the permutation on the integers $1,\dots,|\alpha|$
that permutes the blocks $\{L_{i}\}$ via $\sigma$. In other words,
 $\sigma$ s the permutation 
\[
\left(\begin{array}{ccc}
1 & \dots & n\\
\sigma(1) & \dots & \sigma(n)
\end{array}\right)
\]
 then $\tlist{\alpha}n(\sigma)$ is the permutation defined by writing
\[
\left(\begin{array}{ccc}
L_{1} & \dots & L_{n}\\
L_{\sigma(1)} & \dots & L_{\sigma(n)}
\end{array}\right)
\]
 and regarding the upper and lower rows as sequences length $|\alpha|$. 
\end{enumerate}
\end{defn}
\begin{rem*}
Do not confuse the $T$-maps defined here with the transposition map
for tensor products of chain-complexes. We will use the special notation
$T_{i}$ to represent $T_{1,\dots,2,\dots,1}$, where the 2 occurs
in the $i^{\mathrm{th}}$ position. The two notations don't conflict
since the old notation is never used in the case when $n=1$. Here
is an example of the computation of $\tmap{2,1,3}((1,3,2))=\tmap{2,1,3}\left(\begin{array}{ccc}
1 & 2 & 3\\
3 & 1 & 2
\end{array}\right)$:$L_{1}=\{1\}2$, $L_{2}=\{3\}$, $L_{3}=\{4,5,6\}$. The permutation
maps the ordered set $\{1,2,3\}$ to $\{3,1,2\}$, so we carry out
the corresponding mapping of the sequences $\{L_{1},L_{2},L_{3}\}$
to get $\left(\begin{array}{ccc}
L_{1} & L_{2} & L_{3}\\
L_{3} & L_{1} & L_{2}
\end{array}\right)=\left(\begin{array}{ccc}
\{1,2\} & \{3\} & \{4,5,6\}\\
\{4,5,6\} & \{1,2\} & \{3\}
\end{array}\right)=\left(\begin{array}{cccccc}
1 & 2 & 3 & 4 & 5 & 6\\
4 & 5 & 6 & 1 & 2 & 3
\end{array}\right)$ (or $((1,4)(2,5)(3,6))$, in cycle notation).\end{rem*}
\begin{defn}
\label{def:operad} A sequence of differential graded $\mathbb{Z}$-free
modules, $\{\mathcal{V}_{i}\}$, will be said to form an \emph{operad}
if:
\begin{enumerate}
\item there exists a \emph{unit map} (defined by the commutative diagrams
below) 
\[
\eta:\mathbb{Z}\to\mathcal{V}_{1}
\]

\item for all $i>1$, $\mathcal{V}_{i}$ is equipped with a left action
of $S_{i}$, the symmetric group. 
\item for all $k\ge1$, and $i_{s}\ge0$ there are maps 
\[
\gamma:\mathcal{V}_{i_{1}}\otimes\cdots\otimes\mathcal{V}_{i_{k}}\otimes\mathcal{V}_{k}\to\mathcal{V}_{i}
\]
 where $i=\sum_{j=1}^{k}i_{j}$.

The $\gamma$-maps must satisfy the conditions:

\end{enumerate}
\end{defn}
\begin{description}
\item [{Associativity}] the following diagrams commute, where $\sum j_{t}=j$,
$\sum i_{s}=i$, and $g_{\alpha}=\sum_{\ell=1}^{\alpha}j_{\ell}$
and $h_{s}=\sum_{\beta=g_{s-1}+1}^{g_{s}}i_{\beta}$: \foreignlanguage{american}{
\[
\xyC{50pt}\xyR{15pt}\xymatrix{{\left(\bigotimes_{s=1}^{j}\mathcal{V}_{i_{s}}\right)\otimes\left(\bigotimes_{t=1}^{k}\mathcal{V}_{j_{t}}\right)\otimes\mathcal{V}_{k}}\ar[r]^{{\qquad\text{Id}\otimes\gamma}}\ar[dd]_{\text{shuffle}} & {\left(\bigotimes_{s=1}^{j}\mathcal{V}_{i_{s}}\right)\otimes\mathcal{V}_{j}}\ar[d]^{\gamma}\\
 & {\mathcal{V}_{i}}\\
{\left(\left(\bigotimes_{q=1}^{j_{t}}\mathcal{V}\right)\otimes\bigotimes_{t=1}^{k}\mathcal{V}_{j_{t}}\right)\otimes\mathcal{V}_{k}}\ar[r]_{{\qquad(\otimes_{t}\gamma)\otimes\text{Id}}} & {\left(\bigotimes_{t=1}^{k}\mathcal{V}_{h_{k}}\right)\otimes\mathcal{V}_{k}}\ar[u]_{\gamma}
}
\]
}
\item [{Units}] the diagrams  \[\begin{array}{cc}\xymatrix{{{\integers}^{k}\otimes\mathcal{V}_{k}}\ar[r]^{\cong}\ar[d]_{{\eta}^{k}\otimes\text{Id}}&{\mathcal{V}_{k}}\\
{{\mathcal{V}_{1}}^{k}\otimes{\mathcal{V}_{k}}}\ar[ur]_{\gamma}&}&\xymatrix{{\mathcal{V}_{k}\otimes\integers}\ar[r]^{\cong}\ar[d]_{\text{Id}\otimes\eta}&{\mathcal{V}_{k}}\\
{\mathcal{V}_{k}\otimes\mathcal{V}_{1}}\ar[ur]_{\gamma}&}\end{array}\] commute.
\item [{Equivariance}] the diagrams  \[\xymatrix@C+20pt{{\mathcal{V}_{j_{1}}\otimes\cdots\otimes\mathcal{V}_{j_{k}}\otimes\mathcal{V}_{k}}\ar[r]^-{\gamma}\ar[d]_{\sigma^{-1}\otimes\sigma}&{\mathcal{V}_{j}}\ar[d]^{\tmap{j_{1},\dots,j_{k}}(\sigma)}\\
{\mathcal{V}_{j_{\sigma(1)}}\otimes\cdots\otimes\mathcal{V}_{j_{\sigma(k)}}\otimes\mathcal{V}_{k}}\ar[r]_-{\gamma}&{\mathcal{V}_{j}}}\]commute, where $\sigma\in S_{k}$, and the $\sigma^{-1}$ on the left
permutes the factors $\{\mathcal{V}_{j_{i}}\}$ and the $\sigma$
on the right simply acts on $\mathcal{V}_{k}$. See \ref{def:tmap}
for a definition of $\tmap{j_{1},\dots,j_{k}}(\sigma)$. \[\xymatrix@C+20pt{{\mathcal{V}_{j_{1}}\otimes\cdots\otimes\mathcal{V}_{j_{k}}\otimes\mathcal{V}_{k}}\ar[r]^-{\gamma}\ar[d]_{\tau_{1}\otimes\cdots\tau_{k}\otimes\text{Id}}&{\mathcal{V}_{j}}\ar[d]^-{\tau_{1}\oplus\cdots\oplus\tau_{k}}\\
{\mathcal{V}_{j_{\sigma(1)}}\otimes\cdots\otimes\mathcal{V}_{j_{\sigma(k)}}\otimes\mathcal{V}_{k}}\ar[r]_-{\gamma}&{\mathcal{V}_{j}}}\] where $\tau_{s}\in S_{j_{s}}$ and $\tau_{1}\oplus\cdots\oplus\tau_{k}\in S_{j}$
is the block sum.\end{description}
\begin{rem*}
The alert reader will notice a discrepancy between our definition
of operad and that in \cite{Kriz-May} (on which it was based). The
difference is due to our using operads as parameters for systems of
\emph{maps}, rather than $n$-ary operations. We, consequently, compose
elements of an operad as one composes \emph{maps}, i.e. the second
operand is to the \emph{left} of the first. This is also why the symmetric
groups act on the \emph{left} rather than on the right. \end{rem*}
\begin{defn}
\label{def:unitaloperad}An operad, $\mathcal{V}$, will be called
\emph{unital} if $\mathcal{V}$ has a $0$-component $\mathcal{V}_{0}=\integers$,
concentrated in dimension $0$ and augmentations
\[
\epsilon_{n}:\mathcal{V}_{0}\otimes\cdots\otimes\mathcal{V}_{0}\otimes\mathcal{V}_{n}=\mathcal{V}_{n}\to\mathcal{V}_{0}=\integers
\]
 induced by their structure maps.\end{defn}
\begin{rem*}
The term ``unital operad'' is used in different ways by different
authors. We use it in the sense of Kriz and May in \cite{Kriz-May},
meaning the operad has a $0$-component that acts like an arity-lowering
augmentation under compositions. 
\end{rem*}
We will frequently want to think of operads in other terms:
\begin{defn}
\label{def:operadcomps} Let $\mathcal{V}$ be an operad, as defined
above. Given $i\le k_{1}>0$, define the $i^{\mathrm{th}}$ \emph{composition}

\[
\circ_{i}:\mathcal{V}_{k_{2}}\otimes\mathcal{V}_{k_{1}}\to\mathcal{V}_{k_{1}+k_{2}-1}
\]
 as the composite
\begin{multline*}
\underbrace{\integers\otimes\cdots\otimes\integers\otimes\mathcal{V}_{k_{2}}\otimes\integers\otimes\cdots\otimes\integers}_{\text{\ensuremath{i^{\text{th}}}factor}}\otimes\mathcal{V}_{k_{1}}\\
\to\underbrace{\mathcal{V}_{1}\otimes\cdots\otimes\mathcal{V}_{1}\otimes\mathcal{V}_{k_{2}}\otimes\mathcal{V}_{1}\otimes\cdots\otimes\mathcal{V}_{1}}_{\text{\ensuremath{i^{\text{th}}}factor}}\otimes\mathcal{V}_{k_{1}}\to\mathcal{V}_{k_{1}+k_{2}-1}
\end{multline*}
 where the final map on the right is $\gamma$. 

These compositions satisfy the following conditions, for all $a\in\mathscr{U}_{n}$,
$b\in\mathscr{U}_{m}$, and $c\in\mathscr{U}_{t}$:
\begin{description}
\item [{Associativity}] $(a\circ_{i}b)\circ_{j}c=a\circ_{i+j-1}(b\circ_{j}c)$
\item [{Commutativity}] $a\circ_{i+m-1}(b\circ_{j}c)=(-1)^{mn}b\circ_{j}(a\circ_{i}c)$ 
\item [{Equivariance}] $a\circ_{\sigma(i)}(\sigma\cdot b)=\tunderi ni(\sigma)\cdot(a\circ_{i}b)$ 
\end{description}
\end{defn}
\begin{rem*}
I am indebted to Jim Stasheff for pointing out to me that operads
were originally defined this way and called \emph{composition algebras.}
Given this definition of operad, we recover the $\gamma$ map in definition~\ref{def:operad}
by setting: 
\[
\gamma(u_{i_{1}}\otimes\cdots\otimes u_{i_{k}}\otimes u_{k})=u_{i_{1}}\circ_{1}\cdots\circ_{k-1}u_{i_{k}}\circ_{k}u_{k}
\]
 (where the implied parentheses associate to the right). It is left
to the reader to verify that the two definitions are equivalent (the
commutativity condition, here, is a special case of the equivariance
condition). Given a \emph{unital} operad, we can use the augmentation
maps to recover the composition operations.
\end{rem*}
A simple example of an operad is:
\begin{example}
\label{example:frakS0}For each $n\ge0$, $X$, the operad $\s_{0}$
has $\s_{0}(n)=\integers S_{n}$, concentrated in dimension $0$,
with structure-map induced by
\begin{eqnarray*}
\gamma_{\alpha_{1},\dots,\alpha_{n}}:S_{\alpha_{1}}\times\cdots\times S_{\alpha_{n}}\times S_{n} & \to & S_{\alpha_{1}+\cdots+\alpha_{n}}\\
\sigma_{\alpha_{1}}\times\cdots\times\sigma_{\alpha_{n}}\times\sigma_{n} & \mapsto & \tlist{\alpha}n(\sigma_{n})\circ(\sigma_{\alpha_{1}}\oplus\cdots\oplus\sigma_{\alpha_{n}})
\end{eqnarray*}
In other words, each of the $S_{\alpha_{i}}$ permutes elements within
the subsequence $\{\alpha_{1}+\cdots+\alpha_{i-1}+1,\dots,\alpha_{1}+\cdots+\alpha_{i}\}$
of the sequence $\{1,\dots,\alpha_{1}+\cdots+\alpha_{n}\}$ and $S_{n}$
permutes these $n$ blocks. 
\end{example}
For the purposes of this paper, the main example of an operad is
\begin{defn}
\label{def:coend}Given any $C\in\chaincat$, the associated \emph{coendomorphism
operad}, $\coend(C)$ is defined by
\[
\coend(C)(n)=\homz(C,C^{\otimes n})
\]
 Its structure map
\begin{multline*}
\gamma_{\alpha_{1},\dots,\alpha_{n}}:\homz(C,C^{\otimes n})\otimes\homz(C,C^{\otimes\alpha_{1}})\otimes\cdots\otimes\homz(C,C^{\otimes\alpha_{n}})\to\\
\homz(C,C^{\otimes\alpha_{1}+\cdots+\alpha_{n}})
\end{multline*}
simply composes a map in $\homz(C,C^{\otimes n})$ with maps of each
of the $n$ factors of $C$. 

This is a non-unital operad, but if $C\in\chaincat$ has an augmentation
map $\varepsilon:C\to\ring$ then we can  regard $\epsilon$ as the
only element of $\homz(C,C^{\otimes n})=\homz(C,C^{\otimes0})=\homz(C,\ring)$.
\end{defn}
Morphisms of operads are defined in the obvious way:
\begin{defn}
\label{def:operadmorphism} Given two operads $\mathcal{V}$ and $\mathcal{W}$,
a \emph{morphism} 
\[
f:\mathcal{V}\to\mathcal{W}
\]
 is a sequence of chain-maps 
\[
f_{i}:\mathcal{V}_{i}\to\mathcal{W}_{i}
\]
 commuting with all the diagrams in \ref{def:operad}.
\end{defn}
Verification that this satisfies the required identities is left to
the reader as an exercise.
\begin{defn}
\label{def:sfrakfirstmention}Let $\s$ denote the operad defined
in \cite{Smith:1994} --- where $\s_{n}=\rs n$ is the normalized
bar-resolution of $\integers$ over $\zs n$ for all $n>0$. This
is similar to the Barratt-Eccles operad defined in \cite{Barratt-Eccles-operad},
except that the latter is composed of \emph{unnormalized} bar-resolutions.
See \cite{Smith:1994} or appendix~A of \cite{smith-cellular}, for
the details. 

Appendix~A of \cite{smith-cellular} contains explicit computations
of some composition-operations in $\s$.
\end{defn}
Now we are ready to define the all-important concept of \emph{coalgebras}
over an operad:
\begin{defn}
\label{def:coalgebra-over-operad-1}A chain-complex $C$ is a \emph{coalgebra
over the operad} $\mathcal{V}$ if there exists a morphism of operads
\[
\mathcal{V}\to\coend(C)
\]
\end{defn}
\begin{rem*}
A coalgebra, $C$, over an operad, $\mathcal{V}$, is a sequence of
maps 
\[
f_{n}:\mathcal{V}_{n}\otimes C\to C^{\otimes n}
\]
 for all $n>0$, where $f_{n}$ is $\zs n$-equivariant and $S_{n}$
acts by permuting factors of $C^{\otimes n}$. The maps, $\{f_{n}\}$,
are related in the sense that they fit into commutative diagrams:
\begin{equation}
\xyC{20pt}\xymatrix{{\mathscr{\mathcal{V}}_{n}\otimes\mathscr{\mathcal{V}}_{m}\otimes C}\ar[r]^{\circ_{i}} & {\mathscr{\mathcal{V}}_{n+m-1}\otimes C}\ar[r]^{f_{n+m-1}} & {C^{\otimes n+m-1}}\\
{\mathscr{\mathcal{V}}_{n}\otimes\mathscr{\mathcal{V}}_{m}\otimes C}\ar[r]_{1\otimes f_{m}}\ar@{=}[u] & {\mathscr{\mathcal{V}}_{n}\otimes C^{\otimes m}}\ar[r]_{Z_{i-1}\qquad\quad} & {C^{i-1}\otimes\mathscr{\mathcal{V}}_{n}\otimes C\otimes C^{\otimes m-i}}\ar[u]_{1\otimes\dots\otimes f_{n}\otimes\dots\otimes1}
}
\label{dia:coalgebra-over-operad}
\end{equation}
 for all $n,m\ge1$ and $1\le i\le m$. Here $Z_{i-1}:\mathcal{V}_{n}\otimes C^{m}\to C^{\otimes i-1}\otimes\mathcal{V}_{n}\otimes C\otimes C^{\otimes m-i}$
is the map that shuffles the factor $\mathcal{V}_{n}$ to the right
of $i-1$ factors of $C$. In other words: The abstract composition-operations
in $\mathcal{V}$ exactly correspond to compositions of maps in $\{\homz(C,C^{\otimes n})\}$.
We exploit this behavior in applications of coalgebras over operads,
using an explicit knowledge of the algebraic structure of $\mathcal{V}$.
\end{rem*}
The structure of a coalgebra over an operad can also be described
in several equivalent ways:
\begin{enumerate}
\item $f_{n}:\mathcal{V}(n)\otimes C\to C^{\otimes n}$
\item $g:C\to\prod_{n=0}^{\infty}\homzs n(\mathcal{V}(n),C^{\otimes n})$
\end{enumerate}
where both satisfy identities that describe how composites of these
maps are compatible with the operad-structure.
\begin{defn}
\label{def:coalgebra-over-operad}A chain-complex $C$ is a \emph{coalgebra
over the operad} $\mathcal{V}$ if there exists a morphism of operads
\[
\mathcal{V}\to\coend(C)
\]
\end{defn}
\begin{rem*}
The structure of a coalgebra over an operad can be described in several
equivalent ways:
\begin{enumerate}
\item $f_{n}:\mathcal{V}(n)\otimes C\to C^{\otimes n}$
\item $g:C\to\prod_{n=0}^{\infty}\homzs n(\mathcal{V}(n),C^{\otimes n})$
\end{enumerate}
\end{rem*}
where both satisfy identities that describe how composites of these
maps are compatible with the operad-structure.
\begin{defn}
\label{def:s-coalgebra-morphism}Using the second description,
\[
\alpha_{C}:C\to\prod_{n=1}^{\infty}\homzs n(\rs n,C^{\otimes n})
\]
an $\s$-coalgebra morphism 
\[
f:C\to D
\]
is a chain-map that makes the diagram
\[
\xymatrix{{\forgetful C}\ar[d]_{\forgetful f}\ar[r]^{\alpha_{C}\qquad\qquad\quad} & {\prod_{n=1}^{\infty}\homzs n(\rs n,\forgetful C^{\otimes n})}\ar[d]^{\prod_{n=1}^{\infty}\homzs n(1,\forgetful f^{\otimes n})}\\
{\forgetful D}\ar[r]_{\alpha_{D}\qquad\qquad\quad} & {\prod_{n=1}^{\infty}\homzs n(\rs n,\forgetful D^{\otimes n})}
}
\]
commute, where $\forgetful *$ is the forgetful functor that turns
a coalgebra into a chain-complex.

We also need
\end{defn}

\section{morphisms of $\s$-coalgebras\label{sec:morphisms}}

Proposition~\ref{pro:simplicespropertyS} proves that if $e_{n}=\underbrace{[(1,2)|\cdots|(1,2)]}_{n\text{ terms}}\in\rs 2$
and $x\in\cfn X$ is the image of a $k$-simplex, then
\[
f_{2}(e_{k}\otimes x)=\xi_{k}\cdot x\otimes x
\]
where $\xi_{k}=(-1)^{k(k-1)/2}$.
\begin{defn}
\label{def:gamma-m-map}If \foreignlanguage{american}{$k,m$ are positive
integers, $C$ is a chain-complex, and $F_{2,m}=e_{m}$ and $F_{k,m}=\underbrace{e_{m}\circ_{1}\cdots\circ_{1}e_{m}}_{k-1\text{ iterations}}\in RS_{k}$
--- compositions in the operad $\s$ --- set 
\[
\rho_{m}=(\xi_{m}\cdot E_{2,m},\xi_{m}^{2}\cdot E_{3,m},\xi_{m}^{3}\cdot E_{4,m},\dots)\in\prod_{n=2}^{\infty}RS_{n}
\]
with $\xi_{m}=(-1)^{m(m-1)/2}$ and define 
\[
\gamma_{m}:\prod_{n=2}^{\infty}\homzs n(\rs n,C_{m}^{\otimes n})\to\prod_{n=2}^{\infty}C^{\otimes n}
\]
via evaluation on $\rho_{m}$.}
\end{defn}
We have
\begin{cor}
\label{cor:simplex-image}If $X$ is a simplicial set and $c\in\cf X$
is an element generated by an $n$-simplex, then the image of $c$
under the composite
\[
\cfn X_{n}\xrightarrow{\alpha_{\cf X}}\prod_{k=1}^{\infty}\homzs k(\rs k,\cfn X^{\otimes k})\xrightarrow{\gamma_{n}}\prod_{k=1}^{\infty}\cfn X^{\otimes k}
\]
 is 
\[
e(c)=(c,c\otimes c,\dots)
\]
\end{cor}
\begin{proof}
This follows immediately from proposition~\ref{pro:simplicespropertyS}
and the fact that operad-compositions map to compositions of coproducts.
\end{proof}
Lemma~\ref{lem:diagonals-linearly-independent} implies that
\begin{cor}
\label{cor:n-simplices-map-to-simplices}Let $X$ be a simplicial
set and suppose 
\[
f:\ns n=\cfn{\Delta^{n}}\to\cfn X
\]
 is a $\s$-coalgebra morphism. Then the image of the generator $\Delta^{n}\in\cfn{\Delta^{n}}$
is a generator of $\cfn X$ defined by an $n$-simplex of $X$.\end{cor}
\begin{proof}
Suppose 
\[
f(\Delta^{n})=\sum_{k=1}^{t}c_{k}\cdot\sigma_{k}^{n}\in\cfn X
\]
where the $\sigma_{k}^{n}$ are images of $n$-simplices of $X$.
If $f(\Delta^{n})$ is not equal to one of the $\sigma_{k}^{n}$,
lemma~\ref{lem:diagonals-linearly-independent} implies that its
image is linearly independent of the $\sigma_{k}^{n}$, a \emph{contradiction.}
The statement about sub-simplices follows from the main statement.
\end{proof}
We also conclude that:
\begin{cor}
\label{cor:automorphisms-trivial}If $f:\cfn{\Delta^{n}}\to\cfn{\Delta^{n}}$
is 
\begin{enumerate}
\item an isomorphism of $\s$-algebras in dimension $n$ and 
\item an endomorphism in lower dimensions 
\end{enumerate}
then $f$ must be an isomorphism. If $n\le3$, then $f$ must also
be the identity map.\end{cor}
\begin{rem*}
The final statement actually works for some larger values of $n$,
but the arguments become vastly more complicated (requiring the use
of higher coproducts). It would have extraordinary implications if
it were true for \emph{all} $n$.\end{rem*}
\begin{proof}
We first show that $f$ must be an isomorphism. We are given that
$f$ is an isomorphism in dimension $n$. We use downward induction
on dimension to show that it is an isomorphism in lower dimensions:

Suppose $f$ is an isomorphism in dimension $k$ and $\Delta^{k}\subset\Delta^{n}$
is a simplex. The boundary of $\Delta^{k}$ is a linear combination
of $k+1$ distinct faces which corollary\foreignlanguage{american}{~\ref{cor:n-simplices-map-to-simplices}
implies must map to $k-1$-dimensional simplices with the \emph{same}
coefficients (of $\pm1$). The Pigeonhole Principal and the fact that
$f$ is a \emph{chain-map} imply that all of the $k+1$ distinct faces
of $f(\Delta^{k})$ must be in the image of $f$ so that $f$ induces
a 1-1 correspondence between faces of $\Delta^{k}$ and those of $f(\Delta^{k})$.
It follows that $f|\cfn{\Delta^{k}}$ is an isomorphism in dimension
$k-1$. Since $\Delta^{k}$ was arbitrary, it follows that $f$ is
an isomorphism in dimension $k-1$.}

It follows that $f$ is actually an \emph{automorphism} of $\cfn{\Delta^{n}}$.
Now we assume that $n\le3$ and show that $f$ is the \emph{identity
map:}

If $n=1$ then corollary~\ref{cor:n-simplices-map-to-simplices}
implies that $f|\cfn{\Delta^{1}}_{1}=1$. Since the $0$-simplices
must map to $0$-simplices (by corollary~\ref{cor:n-simplices-map-to-simplices})
with a $+1$ sign it follows that the only possible non-identity automorphism
of $\cfn{\Delta^{1}}$ swaps the ends of $\Delta^{1}$ --- but this
would violate the condition that $f$ is a chain-map.

In dimension 2, let $\Delta^{2}$ be a $2$-simplex. Similar reasoning
to that used in the one-dimensional case implies that a non-identity
automorphism of $\cfn{\Delta^{2}}$ would (at most) involve permuting
some of its faces. Since
\[
\partial\Delta^{2}=F_{0}\Delta^{2}-F_{1}\Delta^{2}+F_{2}\Delta^{2}
\]
the only non-identity permutation compatible with the boundary map
swaps $F_{0}\Delta^{2}$ and $F_{2}\Delta^{2}$. But the coproduct
of $\Delta^{2}$ is given by
\[
[\,]_{2}\otimes\Delta^{2}\mapsto\Delta^{2}\otimes F_{0}F_{1}\Delta^{2}+F_{2}\Delta^{2}\otimes F_{0}\Delta^{2}+F_{1}F_{2}\Delta^{2}\otimes\Delta^{2}
\]
(see proposition~\ref{prop:e1timesdelta2}) where $[\,]$ is the
0-dimensional generator of $\rs 2$ --- the bar-resolution of $\ints$
over $\zs{_{2}}$. It follows that swapping $F_{0}\Delta^{2}$ and
$F_{2}\Delta^{2}$ would violate the condition that $f$ must preserve
coproducts. The case where $n=1$ implies that the vertices cannot
be permuted.

When $n=3$, we have 
\[
\partial\Delta^{3}=F_{0}\Delta^{3}-F_{1}\Delta^{3}+F_{2}\Delta^{3}-F_{3}\Delta^{3}
\]
so, in principal, we might be able to swap $F_{0}\Delta^{3}$ and
$F_{2}\Delta^{3}$ or $F_{1}\Delta^{3}$ and $ $$F_{3}\Delta^{3}$.
The coproduct does not rule any of these actions out since it involves
multiple face-operations. The first ``higher coproduct'' does, however
--- see \ref{prop:e1timesdelta3}:\foreignlanguage{english}{
\begin{align}
f_{2}([(1,2)]\otimes\Delta^{3}) & =F_{1}F_{2}\Delta^{3}\otimes\Delta^{3}-F_{2}\Delta^{3}\otimes F_{0}\Delta^{3}\nonumber \\
 & +\Delta^{3}\otimes F_{0}F_{1}\Delta^{3}-\Delta^{3}\otimes F_{0}F_{3}\Delta^{3}\nonumber \\
 & -F_{1}\Delta^{3}\otimes F_{3}\Delta^{3}-\Delta^{3}\otimes F_{2}F_{3}\Delta^{3}\label{eq:e1timesdelta3-1}
\end{align}
The two terms with two-dimensional factors are $-F_{2}\Delta^{3}\otimes F_{0}\Delta^{3}$
and $-F_{1}\Delta^{3}\otimes F_{3}\Delta^{3}$ and these would be
altered by the permutation mentioned above. It follows that the only
automorphism of $\cfn{\Delta^{3}}$ is the identity map. The lower-dimensional
cases imply that the 1-simplices and vertices cannot be permuted either.}
\end{proof}
A similar line of reasoning implies that:
\begin{cor}
\label{cor:cf-gives-simplices}Let $X$ be a simplicial complex and
let 
\[
f:\cfn{\Delta^{n}}\to\cfn X
\]
map $\Delta^{n}$ to a simplex $\sigma\in\cf X$ defined by the inclusion
$\iota:\Delta^{n}\to X$. Then
\[
f(\cfn{\Delta^{n}})\subset\cfn{\iota}(\cfn{\Delta^{n}})
\]
so that $f=\alpha\circ\cfn{\iota}$, where $\alpha:\cfn{\Delta^{n}}\to\cfn{\Delta^{n}}$
is an automorphism. If $n\le3$, then $f=\cfn{\iota}$.\end{cor}
\begin{proof}
Since $X$ is a simplicial complex, the map $\iota$ is an inclusion.

Suppose $\Delta^{k}\subset\Delta^{n}$ and $f(\cfn{\Delta^{k}})_{k}\subset\cfn{\Delta^{k}}_{k}$.
Since the boundary of $\Delta^{k}$ is an alternating sum of $k+1$
faces, and since they must map to $k-1$-dimensional simplices of
$\cfn{f(\Delta^{k})}$ with the same signs (so no cancellations can
take place) we must have $f(F_{i}\Delta^{k})\subset\cfn{f(\Delta^{k})}$
and the conclusion follows by downward induction on dimension. The
final statements follow immediately from corollary~\ref{cor:automorphisms-trivial}.
\end{proof}

\section{The functor $\nfc *$}

We define a complement to the $\cfn *$-functor: 
\begin{defn}
\label{def:fc}Define a functor
\[
\nfc *:\ircoalgcat\to\spaces
\]
to the category of semi-simplicial sets, as follows:

If $C\in\ircoalgcat$, define the $n$-simplices of $\nfc C$ to be
the $\mathfrak{S}$-coalgebra morphisms
\[
\ns n\to C
\]
where $\ns n=\cfn{\Delta^{n}}$ is the normalized chain-complex of
the standard $n$-simplex, equipped with the $\s$-coalgebra structure
defined in theorem~\ref{thm:ns-construct}.

Face-operations are duals of coface-operations
\[
d_{i}:[0,\dots,i-1,i+1,\dots n]\to[0,\dots,n]
\]
with $i=0,\dots,n$ and vertex $i$ in the target is \emph{not} in
the image of $d_{i}$.\end{defn}
\begin{rem*}
Compare this to the functor $\fc *$ defined in \cite{smith-cellular}.
The subscript $\mathbf{n}$ emphasizes that we do not take \emph{degeneracies}
into account.\end{rem*}
\begin{prop}
\label{prop:ux-map}If $X$ is a simplicial complex (i.e., its simplices
are uniquely determined by their vertices) there exists a natural
map
\[
u_{X}:X\to\nfc{\cfn X}
\]
\end{prop}
\begin{proof}
To prove the first statement, note that any simplex $\Delta^{k}$
in $X$ comes equipped with a canonical inclusion
\[
\iota:\Delta^{k}\to X
\]
The corresponding order-preserving map of vertices induces an $\s$-coalgebra
morphism 
\[
\cfn{\iota}:\cfn{\Delta^{k}}=\ns k\to\cfn X
\]
so $u_{X}$ is defined by
\[
\Delta^{k}\mapsto\cfn{\iota}
\]
It is not hard to see that this operation respects face-operations.\end{proof}
\begin{thm}
\label{thm:simplicial-complexes-determined}If $X\in\simpc$ is a
simplicial complex then the canonical map
\[
u_{X}:X\to\nfc{\cfn X}
\]
defined in proposition~\ref{prop:ux-map} is an isomorphism of the
3-skeleton.\end{thm}
\begin{proof}
This follows immediately from corollary~\ref{cor:n-simplices-map-to-simplices},
which implies that simplices map to simplices and corollary~\ref{cor:cf-gives-simplices},
which implies that these maps are \emph{unique. }\end{proof}
\begin{cor}
\label{cor:cellular-determines-pi1}If $X$ and $Y$ are simplicial
complexes with the property that there exists an isomorphism 
\[
\cfn X\to\cfn Y
\]
then their 3-skeleta are weakly equivalent and 
\[
\pi_{1}(X)\cong\pi_{1}(Y)
\]
\end{cor}
\begin{proof}
Any morphism $g:\cfn X\to\cfn Y$ induces a morphism of simplicial
sets
\[
\fc g:\nfc{\cfn X}\to\nfc{\cfn Y}
\]
and this is an isomorphism (and homeomorphism) of simplicial complexes
if $g$ is an isomorphism. The conclusion follows from theorem~\ref{thm:simplicial-complexes-determined}
which implies that the canonical maps
\begin{align*}
u_{X}:X\to & \nfc{\cfn X}\\
u_{Y}:Y\to & \nfc{\cfn Y}
\end{align*}
are isomorphisms of the 3-skeleta, and the fact that fundamental groups
are determined by the 2-skeleta.
\end{proof}
\appendix

\section{The functor $\cfn *$\label{sec:The-functor-cfn}}

We begin with the elementary but powerful Cartan Theory of Constructions,
originally described in \cite{Cartan3,Cartan4,Cartan5,Cartan6}:
\begin{lem}
\label{lem:cartanconstruction}Let $M_{i}$, $i=1,2$ be DGA-modules,
where:
\begin{enumerate}
\item $M_{1}=A_{1}\otimes N_{1}$, where $N_{1}$ is $\integers$-free and
$A_{1}$ is a DGA-algebra (so $M_{1}$, merely regarded as a DGA-algebra,
is free on a basis equal to a $\integers$-basis of $N_{1}$)
\item $M_{2}$ is a left DGA-module over a DGA-algebra $A_{2}$, possessing 

\begin{enumerate}
\item a sub-DG-module, $N_{2}\subset M_{2}$, such that $\partial_{M_{2}}|N_{2}$
is injective, 
\item a contracting chain-homotopy $\varphi:M_{2}\to M_{2}$ whose image
lies in $N_{2}\subset M_{2}$.
\end{enumerate}
\end{enumerate}

Suppose we are given a chain-map $f_{0}:M_{1}\to M_{2}$ in dimension
$0$ with $f_{0}(N_{1})\subseteq N_{2}$ and want to extend it to
a chain-map from $M_{1}$ to $M_{2}$, subject to the conditions:
\begin{itemize}
\item $f(N_{1})\subseteq N_{2}$
\item $f(a\otimes n)=g(a)\cdot f(n)$, where $g:A_{1}\to A_{2}$ is some
morphism of DG-modules such that $a\otimes n\mapsto g(a)\cdot f(n)$
 is a chain-map.
\end{itemize}

Then the extension $f:M_{1}\to M_{2}$ exists and is unique.

\end{lem}
\begin{rem*}
In applications of this result, the morphism $g$ will often be a
morphism of DGA-algebras, but this is not necessary.

The \emph{existence} of $f$ immediately follows from basic homological
algebra; the interesting aspect of it is its \emph{uniqueness} (not
merely uniqueness up to a chain-homotopy). We will use it repeatedly
to prove associativity conditions by showing that two apparently different
maps satisfying the hypotheses must be \emph{identical}.

The Theory of Constructions formed the cornerstone of Henri Cartan's
elegant computations of the homology and cohomology of Eilenberg-MacLane
spaces in \cite{Cartan11}.\end{rem*}
\begin{proof}
The uniqueness of $f$ follows by induction and the facts that: 
\begin{enumerate}
\item $f$ is determined by its values on $N_{1}$ 
\item the image of the contracting chain-homotopy, $\varphi$, lies in $N_{2}\subset M_{2}$.
\item the boundary map of $M_{2}$ is \emph{injective} on $N_{2}$ (which
implies that there is a \emph{unique} lift of $f$ into the next higher
dimension).
\end{enumerate}
\end{proof}
Now construct a contracting cochain on the normalized chain-complex
of a standard simplex:
\begin{defn}
\label{def:simplex-contracting-cochain}Let $\Delta^{k}$ be a standard
$k$-simplex with vertices $\{[0],\dots,[k]\}$ and $j$-faces $\{[i_{0},\dots,i_{j}]\}$
with $i_{0}<\cdots<i_{j}$ and let $s^{k}$ denote its normalized
chain-complex with boundary map $\partial$. This is equipped with
an augmentation
\[
\epsilon:s^{k}\to\ints
\]
that maps all vertices to $1\in\ints$ and all other simplices to
$0$. Let 
\[
\iota_{k}:\ints\to s^{k}
\]
 denote the map sending $1\in\ints$ to the image of the vertex $[n]$.
Then we have a contracting cochain\textit{\emph{
\begin{equation}
\varphi_{k}([i_{0},\dots,i_{t}]=\left\{ \begin{array}{cc}
(-1)^{t+1}[i_{0},\dots,i_{t},k] & \mathrm{if}\, i_{t}\ne k\\
0 & \mathrm{if}\, i_{t}=k
\end{array}\right.\label{eq:simplex-contracting-cochain}
\end{equation}
and $1-\iota_{k}\circ\epsilon=\partial\circ\varphi_{k}+\varphi_{k}\circ\partial$.}}\end{defn}
\begin{thm}
\label{thm:ns-construct}The normalized chain-complex of $[i_{0},\dots,i_{k}]=\Delta^{k}$
has a $\s$-coalgebra structure that is natural with respect to order-preserving
mappings of vertex-sets
\[
[i_{0},\dots,i_{k}]\to[j_{0},\dots,j_{\ell}]
\]
with $j_{0}\le\cdots\le j_{\ell}$ and $\ell\ge k$. This $\s$-coalgebra
is denoted $\ns k$.

If $X$ is a simplicial complex (a semi-simplicial set whose simplices
are uniquely determined by their vertices), then the normalized chain-complex
of $X$ has a natural $\s$-coalgebra structure 
\[
\cfn X=\dlimit\ns k
\]
 for $\Delta^{n}\in\boldsymbol{\Delta}\downarrow X$ --- the simplex
category of $X$. \end{thm}
\begin{rem*}
The author has a Common LISP program for computing $f_{n}(x\otimes C(\Delta^{k}))$
--- the number of terms is exponential in the dimension of $x$. 

Compare this with the functor $\cf *$ defined in \cite{Smith:1994}
and \cite{smith-cellular}. For simplicial complexes, $\cf X=\cfn X$.\end{rem*}
\begin{proof}
\textit{\emph{If $C=s^{k}=C(\Delta^{k})$ --- the (unnormalized) chain
complex --- we can define a corresponding contracting homotopy on
$C^{\otimes n}$ via 
\begin{align*}
\Phi= & \varphi_{k}\otimes1\otimes\cdots\otimes1+\iota_{k}\circ\epsilon\otimes\varphi_{k}\otimes\cdots\otimes1+\\
 & \cdots+\iota_{k}\circ\epsilon\otimes\cdots\otimes\iota_{k}\circ\epsilon\otimes\varphi_{k}
\end{align*}
where $\varphi_{k}$, $\iota_{k}$, and $\epsilon$ are as in definition~\ref{def:simplex-contracting-cochain}.
}}Above dimension $0$, $\Phi$ is effectively equal to $\varphi_{k}\otimes1\otimes\cdots\otimes1$\textit{\emph{.
Now set $M_{2}=C^{\otimes n}$ and $N_{2}=\img(\Phi)$. In dimension
$0$, we define $f_{n}$ for all $n$ via:
\[
f_{n}(A\otimes[0])=\left\{ \begin{array}{ll}
[0]\otimes\cdots\otimes[0] & \mathrm{if}\, A=[\,]\\
0 & \mathrm{if}\,\dim A>0
\end{array}\right.
\]
This clearly makes $s^{0}$ a coalgebra over $\s$.}}

\textit{\emph{Suppose that the $f_{n}$ are defined below dimension
$k$. Then $\cf{\partial\Delta^{k}}$ is well-defined and satisfies
the conclusions of this theorem. We define $f_{n}(a[a_{1}|\dots|a_{j}]\otimes[0,\dots,k])$
by induction on $j$, requiring that:}}
\begin{condition}[Invariant Condition]
\label{cond:invariant-condition}
\begin{equation}
f_{n}(A(S_{n},1)\otimes s^{k})\subseteq[i_{1},\dots,k]\otimes\mathrm{other}\,\mathrm{terms}\label{eq:invariant-condition}
\end{equation}

--- in other words, the \emph{leftmost factor} must be in $\img\varphi_{k}$.
This is the same as the leftmost factors being ``rearward'' faces
of $\Delta^{k}$.
\end{condition}
Now we set
\begin{eqnarray}
f_{n}(A\otimes s^{k}) & = & \Phi\circ f_{n}(\partial A\otimes s^{k})\nonumber \\
 & + & (-1)^{\dim A}\Phi\circ f_{n}(A\otimes\partial s^{k})\label{eq:high-diag-comp}
\end{eqnarray}
 where $A\in A(S_{n},1)\subset\rs n$ and the term $f_{n}(A\otimes\partial s^{k})$
refers to the coalgebra structure of $\cf{\partial\Delta^{k}}$.

The term $f_{n}(A\otimes\partial s^{k})$ is defined by induction
and diagram~\ref{dia:coalgebra-over-operad} commutes for it. The
term $f_{n}(\partial A\otimes s^{k})$ is defined by induction on
the dimension of $A$ and diagram~\ref{dia:coalgebra-over-operad}
for it as well.

The composite maps in both branches of diagram~\ref{dia:coalgebra-over-operad}
satisfy condition~\ref{cond:invariant-condition} since:
\begin{enumerate}
\item any composite of $f_{n}$-maps will continue to satisfy condition~\ref{cond:invariant-condition}.
\item $\circ_{i}(1\otimes A(S_{n},1)\otimes\cdots\otimes A(S_{m},1))\subseteq1\otimes A(S_{n+m-1},1)$
so that composing an $f_{n}$-map with $\circ_{i}$ results in a map
that still satisfies condition~\ref{cond:invariant-condition}.
\item the diagram commutes in lower dimensions (by induction on $k$)
\end{enumerate}
Lemma~\ref{lem:cartanconstruction} implies that the composites one
gets by following the two branches of diagram~\ref{dia:coalgebra-over-operad}
must be \emph{equal, }so the diagram commutes.

We ultimately get an expression for $f_{n}(x\otimes[0,\dots,k])$
as a sum of tensor-products of sub-simplices of $[0,\dots,k]$ ---
given as ordered lists of vertices.

We claim that this $\s$-coalgebra structure is natural with respect
to ordered mappings of vertices. This follows from the fact that the
only significant property that the vertex $k$ \emph{has} in equation~\ref{eq:simplex-contracting-cochain},
condition~\ref{cond:invariant-condition} and equation~\ref{eq:high-diag-comp}
is that it is the \emph{highest numbered} vertex. 
\end{proof}
We conclude this section some computations of higher coproducts:
\begin{example}
\label{example:e1timesdelta2}If $[0,1,2]=\Delta^{2}$ is a $2$-simplex,
then

\begin{equation}
f_{2}([\,]\otimes\Delta^{2})=\Delta^{2}\otimes F_{0}F_{1}\Delta^{2}+F_{2}\Delta^{2}\otimes F_{0}\Delta^{2}+F_{1}F_{2}\Delta^{2}\otimes\Delta^{2}\label{eq:delta-2-coproduct-1}
\end{equation}
--- the standard (Alexander-Whitney) coproduct --- and

\begin{align*}
f_{2}([(1,2)]\otimes\Delta^{2})= & [0,1,2]\otimes[1,2]-[0,2]\otimes[0,1,2]\\
 & -[0,1,2]\otimes[0,1]
\end{align*}
or, in face-operations

\begin{align}
f_{2}([(1,2)]\otimes\Delta^{2})= & \Delta^{2}\otimes F_{0}\Delta^{2}-F_{1}\Delta^{2}\otimes\Delta^{2}\label{eq:e1timesdelta2-1}\\
 & -\Delta^{2}\otimes F_{2}\Delta^{2}\nonumber 
\end{align}
\end{example}
\begin{proof}
If we write $\Delta^{2}=[0,1,2]$, we get
\[
f_{2}([\,]\otimes\Delta^{2})=[0,1,2]\otimes[2]+[0,1]\otimes[1,2]+[0]\otimes[0,1,2]
\]

To compute $f_{2}([(1,2)]\otimes\Delta^{2})$ we have a version of
equation~\ref{eq:high-diag-comp}:
\begin{align*}
f(e_{1}\otimes\Delta^{2}) & =\Phi_{2}(f_{2}(\partial e_{1}\otimes\Delta^{2})-\Phi_{2}f_{2}(e_{1}\otimes\partial\Delta^{2})\\
 & =-\Phi_{2}(f_{2}((1,2)\cdot[\,]\otimes\Delta^{2})+\Phi_{2}(f_{2}([\,]\otimes\Delta^{2})-\Phi_{2}f_{2}(e_{1}\otimes\partial\Delta^{2})
\end{align*}
Now 
\begin{align*}
\Phi_{2}(1,2)\cdot(f_{2}([\,]\otimes\Delta^{2})= & (\varphi_{2}\otimes1)\bigl([2]\otimes[0,1,2]-[1,2]\otimes[0,1]\\
 & +[0,1,2]\otimes[0]\bigr)\\
 & +(i\circ\epsilon\otimes\varphi_{2})\bigl([2]\otimes[0,1,2]\\
 & -[1,2]\otimes[0,1]+[0,1,2]\otimes[0]\bigr)\\
= & 0
\end{align*}
and
\begin{align*}
\Phi_{2}(f_{2}([\,]\otimes\Delta^{2})= & (\varphi_{2}\otimes1)\bigl([0,1,2]\otimes[2]+[0,1]\otimes[1,2]\\
 & +[0]\otimes[0,1,2]\bigr)\\
= & [0,1,2]\otimes[1,2]-[0,2]\otimes[0,1,2]
\end{align*}
In addition, proposition~\ref{pro:simplicespropertyS} implies that
\begin{align*}
f_{2}(e_{1}\otimes\partial\Delta^{2})= & [1,2]\otimes[1,2]-[0,2]\otimes[0,2]\\
 & +[0,1]\otimes[0,1]
\end{align*}
so that
\[
\Phi_{2}f_{2}(e_{1}\otimes\partial\Delta^{2})=[0,1,2]\otimes[0,1]
\]

We conclude that
\begin{align*}
f_{2}([(1,2)]\otimes\Delta^{2})= & [0,1,2]\otimes[1,2]-[0,2]\otimes[0,1,2]\\
 & -[0,1,2]\otimes[0,1]
\end{align*}
 which implies equation~\ref{eq:e1timesdelta2-1}.
\end{proof}
We end this section with computations of some ``higher coproducts.''
We have a $\zs 2$-equivariant chain-map
\[
f_{2}(\rs 2\otimes C)\to C\otimes C
\]

\begin{prop}
\label{prop:e1timesdelta2}If $\Delta^{2}$ is a $2$-simplex, then:

\begin{equation}
f_{2}([\,]\otimes\Delta^{2})=\Delta^{2}\otimes F_{0}F_{1}\Delta^{2}+F_{2}\Delta^{2}\otimes F_{0}\Delta^{2}+F_{1}F_{2}\Delta^{2}\otimes\Delta^{2}\label{eq:delta-2-coproduct}
\end{equation}
Here $e_{0}=[\,]$ is the 0-dimensional generator of $\rs 2$ and
this is just the standard (Alexander-Whitney) coproduct. 

In addition, we have: 
\begin{align}
f_{2}([(1,2)]\otimes\Delta^{2})= & \Delta^{2}\otimes F_{0}\Delta^{2}-F_{1}\Delta^{2}\otimes\Delta^{2}\label{eq:e1timesdelta2}\\
 & -\Delta^{2}\otimes F_{2}\Delta^{2}\nonumber 
\end{align}
\end{prop}
\begin{proof}
If we write $\Delta^{2}=[0,1,2]$, we get
\[
f_{2}([\,]\otimes\Delta^{2})=[0,1,2]\otimes[2]+[0,1]\otimes[1,2]+[0]\otimes[0,1,2]
\]

To compute $f_{2}([(1,2)]\otimes\Delta^{2})$ we have a version of
equation~\ref{eq:hdiag-comp}:
\begin{align*}
f_{2}(e_{1}\otimes\Delta^{2}) & =\Phi_{2}(f_{2}(\partial e_{1}\otimes\Delta^{2})-\Phi_{2}f_{2}(e_{1}\otimes\partial\Delta^{2})\\
 & =-\Phi_{2}(f_{2}((1,2)\cdot[\,]\otimes\Delta^{2})+\Phi_{2}(f_{2}([\,]\otimes\Delta^{2})-\Phi_{2}f_{2}(e_{1}\otimes\partial\Delta^{2})
\end{align*}
Now 
\begin{align*}
\Phi_{2}(1,2)\cdot(f_{2}([\,]\otimes\Delta^{2})= & (\varphi_{2}\otimes1)([2]\otimes[0,1,2]-[1,2]\otimes[0,1]+[0,1,2]\otimes[0])\\
 & +(i\circ\epsilon\otimes\varphi_{2})([2]\otimes[0,1,2]-[1,2]\otimes[0,1]+[0,1,2]\otimes[0])\\
= & 0
\end{align*}
and
\begin{align*}
\Phi_{2}(f_{2}([\,]\otimes\Delta^{2}) & =(\varphi_{2}\otimes1)\left([0,1,2]\otimes[2]+[0,1]\otimes[1,2]+[0]\otimes[0,1,2]\right)\\
 & =[0,1,2]\otimes[1,2]-[0,2]\otimes[0,1,2]
\end{align*}
In addition, proposition~\ref{pro:simplicespropertyS} implies that
\[
f_{2}(e_{1}\otimes\partial\Delta^{2})=[1,2]\otimes[1,2]-[0,2]\otimes[0,2]+[0,1]\otimes[0,1]
\]
so that
\[
\Phi_{2}f_{2}(e_{1}\otimes\partial\Delta^{2})=[0,1,2]\otimes[0,1]
\]

We conclude that
\begin{align*}
f_{2}([(1,2)]_{2}\otimes\Delta^{2})= & [0,1,2]\otimes[1,2]-[0,2]\otimes[0,1,2]\\
 & -[0,1,2]\otimes[0,1]
\end{align*}
which implies equation~\ref{eq:e1timesdelta2}.
\end{proof}
We continue this computation one dimension higher:
\begin{prop}
\label{prop:e1timesdelta3}If $\Delta^{3}$ is a $3$-simplex, then:
\foreignlanguage{english}{
\begin{align}
f_{2}([(1,2)]\otimes\Delta^{3}) & =F_{1}F_{2}\Delta^{3}\otimes\Delta^{3}-F_{2}\Delta^{3}\otimes F_{0}\Delta^{3}\nonumber \\
 & +\Delta^{3}\otimes F_{0}F_{1}\Delta^{3}-\Delta^{3}\otimes F_{0}F_{3}\Delta^{3}\nonumber \\
 & -F_{1}\Delta^{3}\otimes F_{3}\Delta^{3}-\Delta^{3}\otimes F_{2}F_{3}\Delta^{3}\label{eq:e1timesdelta3}
\end{align}
}\end{prop}
\begin{proof}
As before, $\Delta^{3}=[0,1,2,3]$, and we have 
\begin{align*}
f_{2}(e_{1}\otimes\Delta^{3}) & =\Phi_{3}(f_{2}(\partial e_{1}\otimes\Delta^{3})-\Phi_{3}f_{2}(e_{1}\otimes\partial\Delta^{3})\\
 & =-\Phi_{3}(f_{2}((1,2)\cdot[\,]\otimes\Delta^{3})+\Phi_{3}(f_{2}([\,]\otimes\Delta^{3})-\Phi_{3}f_{2}(e_{1}\otimes\partial\Delta^{3})
\end{align*}
and 
\[
\Phi_{3}(f_{2}((1,2)\cdot[\,]\otimes\Delta^{3})=0
\]
 We also conclude \foreignlanguage{english}{
\begin{align*}
\Phi_{3}(f_{2}([\,]\otimes\Delta^{3})= & [0,3]\otimes\Delta^{3}-[0,1,3]\otimes[1,2,3]\\
 & +\Delta^{3}\otimes[2,3]
\end{align*}
Now
\[
\partial\Delta^{3}=[1,2,3]-[0,2,3]+[0,1,3]-[0,1,2]
\]
and equation~\ref{eq:e1timesdelta2} implies that
\begin{align*}
f_{2}(e_{1}\otimes\partial\Delta^{3}) & =[1,2,3]\otimes[2,3]-[1,3]\otimes[1,2,3]\\
 & -[1,2,3]\otimes[1,2]\\
 & -[0,2,3]\otimes[2,3]+[0,3]\otimes[0,2,3]\\
 & +[0,2,3]\otimes[0,2]\\
 & +[0,1,3]\otimes[1,3]-[0,3]\otimes[0,1,3]\\
 & -[0,1,3]\otimes[0,1]\\
 & -[0,1,2]\otimes[1,2]+[0,2]\otimes[0,1,2]\\
 & +[0,1,2]\otimes[0,1]
\end{align*}
\emph{Most} of these terms die when one applies $\Phi_{3}$:
\begin{align*}
\Phi_{3}f_{2}(e_{1}\otimes\partial\Delta^{3}) & =\Delta^{3}\otimes[1,2]+[0,2,3]\otimes[0,1,2]\\
 & -\Delta^{3}\otimes[0,1]
\end{align*}
We conclude that
\begin{align*}
f_{2}(e_{1}\otimes\Delta^{3}) & =[0,3]\otimes\Delta^{3}-[0,1,3]\otimes[1,2,3]\\
 & +\Delta^{3}\otimes[2,3]-\Delta^{3}\otimes[1,2]\\
 & -[0,2,3]\otimes[0,1,2]-\Delta^{3}\otimes[0,1]
\end{align*}
which implies equation~\ref{eq:e1timesdelta3}. }
\end{proof}
With this in mind, note that images of simplices in $\cfn *$ have
an interesting property:
\begin{prop}
\label{pro:simplicespropertyS}Let $X$ be a simplicial set with $C=\cfn X$
and with coalgebra structure 
\[
f_{n}:RS_{n}\otimes\cfn X\to\cfn X^{\otimes n}
\]
and suppose $RS_{2}$ is generated in dimension $n$ by $e_{n}=\underbrace{[(1,2)|\cdots|(1,2)]}_{n\text{ terms}}$.
If $x\in C$ is the image of a $k$-simplex, then
\[
f_{2}(e_{k}\otimes x)=\xi_{k}\cdot x\otimes x
\]
where $\xi_{k}=(-1)^{k(k-1)/2}$.\end{prop}
\begin{rem*}
This is just a chain-level statement that the Steenrod operation $\operatorname{Sq}^{0}$
acts trivially on mod-$2$ cohomology. A weaker form of this result
appeared in \cite{Davis:mco}.\end{rem*}
\begin{proof}
Recall that $(\rs 2)_{n}=\ints[\ints_{2}]$ generated by $ $$e_{n}=[\underbrace{(1,2)|\cdots|(1,2)}_{n\text{ factors}}]$.
Let $T$ be the generator of $\ints_{2}$ --- acting on $C\otimes C$
by swapping the copies of $C$.

We assume that $f_{2}(e_{i}\otimes C(\Delta^{j}))\subset C(\Delta^{j})\otimes C(\Delta^{j})$
so that 
\begin{equation}
i>j\implies f_{2}(e_{i}\otimes C(\Delta^{j}))=0\label{eq:big-diag-condition}
\end{equation}

\end{proof}
As in section~4 of \cite{Smith:1994}, if $e_{0}=[\,]\in\rs 2$ is
the $0$-dimensional generator, we define
\[
f_{2}:\rs 2\otimes C\to C\otimes C
\]
 inductively by
\begin{eqnarray}
f_{2}(e_{0}\otimes[i]) & = & [i]\otimes[i]\nonumber \\
f_{2}(e_{0}\otimes[0,\dots,k]) & = & \sum_{i=0}^{k}[0,\dots,i]\otimes[i,\dots,k]\label{eq:big-diag1}
\end{eqnarray}
Let $\sigma=\Delta^{k}$ and inductively define
\begin{align}
f_{2}(e_{k}\otimes\sigma) & =\Phi_{k}(f_{2}(\partial e_{k}\otimes\sigma)+(-1)^{k}\Phi_{k}f_{2}(e_{k}\otimes\partial\sigma)\nonumber \\
 & =\Phi_{k}(f_{2}(\partial e_{k}\otimes\sigma)\label{eq:hdiag-comp}
\end{align}
because of equation~\ref{eq:big-diag-condition}. 
\begin{proof}
Expanding $\Phi_{k}$, we get
\begin{align}
f_{2}(e_{k}\otimes\sigma) & =(\varphi_{k}\otimes1)(f_{2}(\partial e_{k}\otimes\sigma))+(i\circ\epsilon\otimes\varphi_{k})f_{2}(\partial e_{k}\otimes\sigma)\nonumber \\
 & =(\varphi_{k}\otimes1)(f_{2}(\partial e_{k}\otimes\sigma))\label{eq:big-diag2}
\end{align}
 because $\varphi_{k}^{2}=0$ and $\varphi_{k}\circ i\circ\epsilon=0$. 

Noting that $\partial e_{k}=(1+(-1)^{k}T)e_{k-1}\in\rs 2$, we get
\begin{align*}
f_{2}(e_{k}\otimes\sigma) & =(\varphi_{k}\otimes1)(f_{2}(e_{k-1}\otimes\sigma)+(-1)^{k}(\varphi_{k}\otimes1)\cdot T\cdot f_{2}(e_{k-1}\otimes\sigma)\\
 & =(-1)^{k}(\varphi_{k}\otimes1)\cdot T\cdot f_{2}(e_{k-1}\otimes\sigma)
\end{align*}
again, because $\varphi_{k}^{2}=0$ and $\varphi_{k}\circ\iota_{k}\circ\epsilon=0$.
We continue, using equation~\ref{eq:big-diag2} to compute $f(e_{k-1}\otimes\sigma)$:
\begin{align*}
f_{2}(e_{k}\otimes\sigma)= & (-1)^{k}(\varphi_{k}\otimes1)\cdot T\cdot f_{2}(e_{k-1}\otimes\sigma)\\
= & (-1)^{k}(\varphi_{k}\otimes1)\cdot T\cdot(\varphi_{k}\otimes1)\biggl(f_{2}(\partial e_{k-1}\otimes\sigma)\\
 & +(-1)^{k-1}f_{2}(e_{k-1}\otimes\partial\sigma)\biggr)\\
= & (-1)^{k}\varphi_{k}\otimes\varphi_{k}\cdot T\cdot\biggl(f_{2}(\partial e_{k-1}\otimes\sigma)\\
 & +(-1)^{k-1}f_{2}(e_{k-1}\otimes\partial\sigma)\biggr)
\end{align*}
If $k-1=0$, then the left term vanishes. If $k-1=1$ so $\partial e_{k-1}$
is $0$-dimensional then equation~\ref{eq:big-diag1} gives $f(\partial e_{1}\otimes\sigma)$
and this vanishes when plugged into $\varphi_{k}\otimes\varphi_{k}$.
If $k-1>1$, then $f_{2}(\partial e_{k-1}\otimes\sigma)$ is in the
image of $\varphi_{k}$, so it vanishes when plugged into $\varphi_{k}\otimes\varphi_{k}$.

In \emph{all} cases, we can write
\begin{align*}
f_{2}(e_{k}\otimes\sigma) & =(-1)^{k}\varphi_{k}\otimes\varphi_{k}\cdot T\cdot(-1)^{k-1}f_{2}(e_{k-1}\otimes\partial\sigma)\\
 & =-\varphi_{k}\otimes\varphi_{k}\cdot T\cdot f_{2}(e_{k-1}\otimes\partial\sigma)
\end{align*}
If $f_{2}(e_{k-1}\otimes\Delta^{k-1})=\xi_{k-1}\Delta^{k-1}\otimes\Delta^{k-1}$
(the inductive hypothesis), then 
\begin{multline*}
f_{2}(e_{k-1}\otimes\partial\sigma)=\\
\sum_{i=0}^{k}\xi_{k-1}\cdot(-1)^{i}[0,\dots,i-1,i+1,\dots k]\otimes[0,\dots,i-1,i+1,\dots k]
\end{multline*}
and the only term that does not get annihilated by $\varphi_{k}\otimes\varphi_{k}$
is 
\[
(-1)^{k}[0,\dots,k-1]\otimes[0,\dots,k-1]
\]
 (see equation~\foreignlanguage{english}{\ref{eq:simplex-contracting-cochain}}).
We get
\begin{align*}
f_{2}(e_{k}\otimes\sigma) & =\xi_{k-1}\cdot\varphi_{k}\otimes\varphi_{k}\cdot T\cdot(-1)^{k-1}[0,\dots,k-1]\otimes[0,\dots,k-1]\\
 & =\xi_{k-1}\cdot\varphi_{k}\otimes\varphi_{k}(-1)^{(k-1)^{2}+k-1}[0,\dots,k-1]\otimes[0,\dots,k-1]\\
 & =\xi_{k-1}\cdot(-1)^{(k-1)^{2}+2(k-1)}\varphi[0,\dots,k-1]\otimes\varphi[0,\dots,k-1]\\
 & =\xi_{k-1}\cdot(-1)^{k-1}[0,\dots,k]\otimes[0,\dots,k]\\
 & =\xi_{k}\cdot[0,\dots,k]\otimes[0,\dots,k]
\end{align*}
where the sign-changes are due to the Koszul Convention. We conclude
that $\xi_{k}=(-1)^{k-1}\xi_{k-1}$.
\end{proof}

\section{Proof of lemma~\ref{lem:diagonals-linearly-independent}}
\begin{lem}
\label{lem:diagonals-linearly-independent}Let $C$ be a free abelian
group, let 
\[
\hat{C}=\ints\oplus\prod_{i=1}^{\infty}C^{\otimes i}
\]

Let $e:C\to\hat{C}$ be the function that sends $c\in C$ to
\[
(1,c,c\otimes c,c\otimes c\otimes c,\dots)\in\hat{C}
\]
For any integer $t>1$ and any set $\{c_{1},\dots,c_{t}\}\in C$ of
distinct, nonzero elements, the elements 
\[
\{e(c_{1}),\dots,e(c_{t})\}\in\rats\otimes_{\ints}\hat{C}
\]
are linearly independent over $\rats$. It follows that $e$ defines
an injective function
\[
\bar{e}:\ints[C]\to\hat{C}
\]
\end{lem}
\begin{proof}
We will construct a vector-space morphism
\begin{equation}
f:\rats\otimes_{\ints}\hat{C}\to V\label{eq:diagonals-linearly-independent1}
\end{equation}
such that the images, $\{f(e(c_{i}))\}$, are linearly independent.
We begin with the ``truncation morphism''
\[
r_{t}:\hat{C}\to\ints\oplus\bigoplus_{i=1}^{t-1}C^{\otimes i}=\hat{C}_{t-1}
\]
which maps $C^{\otimes1}$ isomorphically. If $\{b_{i}\}$ is a $\ints$-basis
for $C$, we define a vector-space morphism 
\[
g:\hat{C}_{t-1}\otimes_{\ints}\rats\to\rats[X_{1},X_{2},\dots]
\]
by setting
\[
g(c)=\sum_{\alpha}z_{\alpha}X_{\alpha}
\]
where $c=\sum_{\alpha}z_{\alpha}b_{\alpha}\in C\otimes_{\ints}\rats$,
and extend this to $\hat{C}_{t-1}\otimes_{\ints}\rats$ via 
\[
g(c_{1}\otimes\cdots\otimes c_{j})=g(c_{1})\cdots g(c_{j})\in\rats[X_{1},X_{2},\dots]
\]
The map in equation~\ref{eq:diagonals-linearly-independent1} is
just the composite
\[
\hat{C}\otimes_{\ints}\rats\xrightarrow{r_{t-1}\otimes1}\hat{C}_{t-1}\otimes_{\ints}\rats\xrightarrow{g}\rats[X_{1},X_{2},\dots]
\]
It is not hard to see that 
\[
p_{i}=f(e(c_{i}))=1+f(c_{i})+\cdots+f(c_{i})^{t-1}\in\rats[X_{1},X_{2},\dots]
\]
 for $i=1,\dots,t$. Since the $f(c_{i})$ are \emph{linear} in the
indeterminates $X_{i}$, the degree-$j$ component (in the indeterminates)
of $f(e(c_{i}))$ is precisely $f(c_{i})^{j}$. It follows that a
linear dependence-relation
\[
\sum_{i=1}^{t}\alpha_{i}\cdot p_{i}=0
\]
with $\alpha_{i}\in\rats$, holds if and only if
\[
\sum_{i=1}^{t}\alpha_{i}\cdot f(c_{i})^{j}=0
\]
 for all $j=0,\dots,t-1$. This is equivalent to $\det M=0$, where
\[
M=\left[\begin{array}{cccc}
1 & 1 & \cdots & 1\\
f(c_{1}) & f(c_{2}) & \cdots & f(c_{t})\\
\vdots & \vdots & \ddots & \vdots\\
f(c_{1})^{t-1} & f(c_{2})^{t-1} & \cdots & f(c_{t})^{t-1}
\end{array}\right]
\]
 Since $M$ is the transpose of the Vandermonde matrix, we get
\[
\det M=\prod_{1\le i<j\le t}(f(c_{i})-f(c_{j}))
\]
Since $f|C\otimes_{\ints}\rats\subset\hat{C}\otimes_{\ints}\rats$
is \emph{injective,} it follows that this \emph{only} vanishes if
there exist $i$ and $j$ with $i\ne j$ and $c_{i}=c_{j}$. The second
conclusion follows.
\end{proof}
\bibliographystyle{amsplain}

\providecommand{\bysame}{\leavevmode\hbox to3em{\hrulefill}\thinspace}
\providecommand{\MR}{\relax\ifhmode\unskip\space\fi MR }
\providecommand{\MRhref}[2]{%
  \href{http://www.ams.org/mathscinet-getitem?mr=#1}{#2}
}
\providecommand{\href}[2]{#2}


    \end{document}